\documentclass[12pt,a4paper]{amsart}
\usepackage[utf8]{inputenc}
\usepackage[sc,osf]{mathpazo}
\usepackage[scaled=0.2]{beramono}
\usepackage[euler-digits]{eulervm}
\usepackage{mathrsfs}
\usepackage{amssymb}
\usepackage{amsfonts}
\usepackage{dsfont}
\usepackage{mathtools}
\usepackage{xfrac}

\usepackage[a4paper,margin=2.5cm,centering]{geometry}
\usepackage[pdfdisplaydoctitle,colorlinks,breaklinks,urlcolor=blue,linkcolor=blue,citecolor=blue]{hyperref}
\usepackage[nameinlink,capitalise]{cleveref}
\newtheorem{theorem}{Theorem}[section]
\newtheorem{lemma}[theorem]{Lemma}
\newtheorem{proposition}[theorem]{Proposition}
\newtheorem{corollary}[theorem]{Corollary}

\newtheorem{conjecture*}{Conjecture}
\theoremstyle{definition}
\newtheorem{definition}[theorem]{Definition}
\newtheorem{assumption}[theorem]{Assumption}
\theoremstyle{remark}
\newtheorem{remark}[theorem]{Remark}
\newtheorem*{remark*}{Remark}
\numberwithin{equation}{section}

\newcommand{\one}{\mathbb{1}}

\newcommand{\norm}[1]{\left\lVert #1 \right\rVert}
 
\newcommand{\R}     {\mathbb{R}} 
 
\newcommand{\N}     {\mathbb{N}}

\newcommand{\E}     {\mathbb{E}}

\newcommand{\bX}    {{\mathbf X }}
\newcommand{\bY}    {{\mathbf Y }}
\newcommand{\bS}    {{\mathbf S }}

\newcommand{\Fcal}   {{\mathcal F }}

\newcommand{\Kcal}   {{\mathcal K }}

\newcommand{\Ncal}   {{\mathcal N }}

\newcommand{\Rcal}   {{\mathcal R }}

\newcommand{\Wcal}   {{\mathcal W }}

\newcommand{\Zcal}   {{\mathcal Z }}

\usepackage{xcolor}
\usepackage[obeyFinal,textwidth=2.7cm]{todonotes}
\usepackage{framed}
\definecolor{mrshadecolor}{rgb}{.9,.9,.9}

\begin{document}
\title[Universality in deep neural networks via the Lindeberg principle]{Universality in Deep Neural Networks:\\
An approach via the Lindeberg exchange principle}
\author{Filippo Giovagnini, Sotirios Kotitsas, Marco Romito}
\begin{abstract}
    We consider the infinite-width limit of a fully connected deep neural network with general weights, and we prove quantitative general bounds on the $2$-Wasserstein distance between the network and its infinite-width Gaussian limit, under appropriate regularity assumptions on the activation function. Our main tool is a Lindeberg principle for Deep Neural Networks, which we use to successively replace the weights on each layer by Gaussian random variables.
\end{abstract}
\maketitle
\tableofcontents


\section{Introduction}

Neural networks have become nowadays the central engines of modern technology, driving the most recent scientific breakthroughs, as well as reshaping how we work and create. While their empirical success is now well established, there is still a significant amount of work to be done to achieve a complete theoretical understanding of their properties. Randomly initialized networks are just as vital, acting as both a practical "blank slate" that allows neurons to learn distinct features and a theoretical gold mine for researchers studying the mathematical foundations of how deep learning actually works. 

The main subject of this paper is the analysis of fully connected deep neural networks, where weights and biases are random variables, and the interest is in their behavior as the widths of the hidden layers tend to infinity. In this regime, random neural networks are expected to display Gaussian behavior: under the standard mean-field scaling of the weights, the output of the network converges to a Gaussian process whose covariance is determined recursively by the architecture and the activation function. This infinite-width description provides a tractable probabilistic model for wide networks and has become a basic tool in the analysis of random neural networks.

Neal's seminal work \cite{neal1996priors} laid the groundwork for the subject by proving the convergence of randomly initialised shallow wide fully connected networks to Gaussian processes. Deep networks were later covered by \cite{lee2018deep,matthews2018gaussian,yang_2019_tensor,yang2019scaling}, with weaker assumptions \cite{Hanin2023}, or with a functional perspective in \cite{BracaleFavaroFortiniPeluchetti2021}. 

Further results focused on quantitative bounds for convergence of randomly initialised networks with Gaussian weights. The first quantitative proof of convergence, with Gaussian initialised weights, was given in \cite{basteri_2024_quantitative}, in Wasserstein-2 distance. The role of the depth-to-width ratio in the convergence to the neural network Gaussian process was introduced in \cite{randomfullyconnectedneuralcumulants}. Further results of quantitative convergence were obtained in \cite{apollonio_2024_normal}, with total variation and Wasserstein-1 distances. Improved convergence rates were obtained in \cite{MR4898096} for Wasserstein-2 distance, and later in \cite{trevisan2023widedeepneuralnetworks} for Wasserstein-$p$ distances. A recent extension to the total variation metric is given in \cite{entr_bounds_cond_gaussian}. Quantitative convergence to Gaussian processes in the infinite width limit turns out to be stable also for trained shallow neural networks \cite{agazzi_2026_quantitative}.

\cite{Hanin2023} provides the first proof of universality of the neural network Gaussian process, namely of convergence of randomly initialised networks with \emph{non-Gaussian} weights. The proof in \cite{Hanin2023} is qualitative, earlier quantitative results of convergence were obtained in \cite{eldan_2021_non,klukowski_2022_rate}. More refined rates of convergence in Wasserstein-1 are given in \cite{balasubramanian_2024_gaussian}, and improved in \cite{Unfortunate_paper} and \cite{celli2026wideneuralnetworksgeneral}. We refer to \cite{celli2026wideneuralnetworksgeneral} for a detailed literature review on the subject, with overviews on both Gaussian and non-Gaussian weights.

The main goal of this paper is to study quantitative rates of convergence for universality of random neural networks. More precisely, we prove explicit bounds on the distance between the output of a finite-width network and its infinite-width Gaussian limit in the $2$-Wasserstein distance. Our approach is based on the Lindeberg exchange principle. Roughly speaking, the weights of the network are successively replaced by Gaussian weights with matching first and second moments and the cumulative error along this switching procedure is carefully estimated. The additional difficulty is that a naive application of the classical Lindeberg method leads to dimension-dependent bounds that are too large in our setting. The key point is that the recursive structure of the network provides additional smoothing, which can be exploited to obtain better estimates.

In contrast to the recent results \cite{Unfortunate_paper} and \cite{celli2026wideneuralnetworksgeneral}, which focus on the 1-Wasserstein, total variation or convex metrics under Lipschitz assumptions on the activation function, our work establishes quantitative bounds in the 2-Wasserstein distance. While our approach necessitates higher regularity for the activation function -- specifically $C^3_b$ in the presence of biases and even higher-order derivatives for the zero-bias case -- we obtain a convergence rate of approximately $O(n^{-1/4})$ ($n$ is layers width), which is more robust than \cite{Unfortunate_paper}, although is slower than the rate obtained in \cite{celli2026wideneuralnetworksgeneral}. This provides a stronger metric of convergence than the $W_1$ results in \cite{Unfortunate_paper} and \cite{celli2026wideneuralnetworksgeneral}, while requiring only finite third moments for the hidden layer weights. Additionally, in the presence of Gaussian biases we are able to recover bounds in the total variation and \(1-\)Wasserstein distance establised in \cite{celli2026wideneuralnetworksgeneral}, with optimal rates.
\bigskip

The rest of the paper is organized as follows: In \textbf{Section \ref{sec:Setup_plus_results}} we present our assumptions and our main results, summarized in \textbf{Theorems \ref{thm:W_2_bound}}, \textbf{\ref{thm:Improved_W_2_bound}} and \textbf{\ref{thm:main-weak-bound}}, while in \textbf{Section \ref{sec:Proof_sketch}} we give a rough outline of our arguments. In \textbf{Section \ref{sec:Lindeberg_bounds}} we prove our main technical estimates and prove \textbf{Theorem \ref{thm:main-weak-bound}}, while in \textbf{Section \ref{sec:W_2_bounds_proofs}} we prove \textbf{Theorems \ref{thm:W_2_bound}} and \textbf{\ref{thm:Improved_W_2_bound}}. Finally, there are two appendices that contain numerical simulations on the distance of the neural network from its infinite width limit, and a proof of a bound on the \(2-\)Wasserstein distance between conditionally independent Gaussian random variables, which we include for completeness.

\section{The setup and main results}\label{sec:Setup_plus_results}

We begin by introducing the model.

\begin{definition}\label{def:Deep_Neural_Network}
    Fix \(L\in\N\) and \(n_0,\dots,n_{L+1}\in\N\). A \textbf{fully connected deep neural network} is a function
    \[
        z^{(L+1)}:\R^{n_0}\to\R^{n_{L+1}}
    \]
    defined recursively by
    \begin{equation}\label{eq:network_recursion}
        z^{(\ell)}=\begin{cases}
                    W^{(1)}x+b^{(1)}, & \ell=1,\\
                    W^{(\ell)}\sigma(z^{(\ell-1)})+b^{(\ell)}, & \ell=2,\dots,L+1.
                \end{cases}
    \end{equation}
    The matrices \(W^{(\ell)}\in\R^{n_l\times n_{l-1}}\) are called \textbf{weights}, and the vectors \(b^{(\ell)}\in\R^{n_l}\) are called \textbf{biases}. Here \(\sigma:\R\to\R\) is the activation function, and for a vector \(x\in\R^d\) we write
    \[
        \sigma(x)=(\sigma(x_1),\dots,\sigma(x_d))\in\R^d.
    \]
    The integer \(n_0\) is the input dimension and \(n_{L+1}\) is the output dimension of the network. We refer to \(z^{(L+1)}\) as the output of the network.
\end{definition}

Following \cite{MR4898096,randomfullyconnectedneuralcumulants}, we consider \textbf{random} neural networks, in the sense that the weights and biases are random. We study the behavior of the random output \(z^{(L+1)}\) as the widths of the hidden layers tend to infinity, that is, as \(n_1,\dots,n_L\to\infty\).

To obtain a non-trivial limit, one needs to impose some assumptions on the weights and biases.

\begin{assumption}\label{assum:weights}
    Let \(\mu\) be a probability measure on \(\R\) with mean zero, variance \(1\), and finite third moments. For \(\ell=2,\dots,L+1\), we choose the weights as
    \begin{equation}\label{eq:weight_assum}
        w^{(\ell)}_{i,j}\sim \sqrt{\frac{C_W}{n_{\ell-1}}}\cdot \mu,
    \end{equation}
    where \(C_W>0\), and all these weights are sampled independently. For \(\ell=1\), the weights \(w_{i,j}^{(1)}\) are all independent Gaussian random variables, independent also of the weights in the subsequent layers, with mean \(0\) and variance \(C_w/n_0\).
\end{assumption}

The assumptions on the biases are

\begin{assumption}\label{assum:biases}
    The biases are Gaussian and i.i.d.. In particular
    \begin{equation}\label{eq:Gaussian_biases}
    b^{(\ell)}\sim \Ncal(0,c_b I_{n_l\times n_l}),
\end{equation}
where \(c_b>0\).
\end{assumption}

Under these assumptions, the infinite-width limit of the network is well understood \cite{Hanin2023}: as \(n_1,\dots,n_L\to\infty\), the output converges to a Gaussian process. In particular

\begin{theorem}\label{thm:qualitative_thm}
    Let \(L,n_0,n_{L+1},r\in\N\) be fixed, and assume that \(\sigma\) is  polynomially bounded
    to order \(r\) in the sense of \cite[Definition \(2.1\)]{MR4898096}. Under \textbf{Assumptions \ref{assum:weights}} and \textbf{\ref{assum:biases}}, the process
    \[
        x\mapsto z^{(L+1)}(x)
    \]
    converges weakly in \(C^{r-1}(\R^{n_0},\R^{n_{L+1}})\) to a Gaussian process
    \[
        x\mapsto \Zcal^{(L+1)}(x),
    \]
    whose \(n_{L+1}\) components are i.i.d.\ and whose covariance structure is defined recursively by
    \[
        K^{(\ell+1)}(x,y)=
        \begin{cases}
            c_b+C_W \E[\sigma(\Zcal^{(\ell)}(x))\sigma(\Zcal^{(\ell)}(y))], & \ell\geq 1,\\[0.3em]
            c_b+\dfrac{C_W}{n_0}x\cdot y, & \ell=0.
        \end{cases}
    \]
    The same convergence is true when the DNN \(z^{(L+1)}\) has zero biases and weights satisfying \textbf{Assumption \ref{assum:weights}}. The limiting Gaussian process is defined recursively as above, with \(c_b=0\).
\end{theorem}

Our main focus is on the speed of convergence of \(z^{(L+1)}\) to \(\Zcal^{(L+1)}\) in suitable probability metrics. In the case of Gaussian weights, one can use techniques from Gaussian analysis and Stein's method to derive quantitative estimates in total variation and in \(p\)-Wasserstein distance; see, for instance, \cite{MR4898096,entr_bounds_cond_gaussian}. For general weight distributions with sufficiently many moments, quantitative limit theorems in total variation and \(1\)-Wasserstein distance were obtained in \cite{celli2026wideneuralnetworksgeneral}. Here we are interested in obtaining bounds in the \(2\)-Wasserstein distance, and we consider two distinct cases: when the network has zero biases and when the network has biases satisfying \textbf{Assumption \ref{assum:biases}}. In the first case, the main result is the following estimate.

\begin{theorem}\label{thm:W_2_bound}
    Let \(\sigma\in C_b^{3\cdot 2^{L-1}}(\R)\), and consider a DNN of depth \(L\) with weights as in \textbf{Assumption \ref{assum:weights}} and zero biases. Let \(\Zcal^{(L+1)}(x)\) be the zero-bias Gaussian process from \textbf{Theorem \ref{thm:qualitative_thm}}, and assume that 
    \(K^{(\ell)}(x)\) is invertible for \(\ell=1,\ldots,n_{L+1}\). Then
    \[
         \Wcal_2\bigl(z^{(L+1)}(x),\Zcal^{(L+1)}(x)\bigr)
         \leq C\biggl(\frac{1}{\sqrt{n_L}}+\Bigl( \sum_{k=1}^{L-1}\frac{1}{\sqrt{n_k}} \Bigr)^{1/4}\biggr),
    \]
    for some constant \(C= C(L,n_{L+1},\sigma)\).
\end{theorem}

Observe that \textbf{Theorem \ref{thm:W_2_bound}} requires rather strong regularity assumptions on the activation function \(\sigma\).  Nevertheless, many activation functions of interest satisfy these hypotheses, for example, the tanh, the arctan or the sigmoid activation function.
The reason for the regularity assumptions is outlined in \textbf{Section \ref{sec:Proof_sketch}}. On the other hand, if the neural network has Gaussian biases, we can leverage the extra smoothing obtained by them to dramatically reduce the regularity required by the activation function. In particular, our second main result is the following theorem

\begin{theorem}\label{thm:Improved_W_2_bound}
    Let \(z^{(L+1)}(x)\) be the output of a DNN with weights and biases satisfying \textbf{Assumptions \ref{assum:weights}} and \textbf{\ref{assum:biases}} respectively. Moreover, assume that \(\sigma\in C^{3}_b(\R)\) and that 
    \(K^{(\ell)}(x)\) is invertible for \(\ell=1,\ldots,n_{L+1}\). Then for some constant \(C= C(L,n_{L+1},\sigma)\)
   
    \[
        \Wcal_2\bigl(z^{(L+1)}(x),\Zcal^{(L+1)}(x)\bigr)
         \leq C\biggl(\frac{1}{\sqrt{n_L}}+\Bigl( \sum_{k=1}^{L-1}\frac{1}{\sqrt{n_k}} \Bigr)^{1/4}\biggr),
    \]
    where now \(\Zcal^{(L+1)}(x)\) is the biased version of the Gaussian process from \textbf{Theorem \ref{thm:qualitative_thm}}.
\end{theorem}

The Gaussianity of the biases is essential in the proof of this improved result, just as it is in the proof of \textbf{Theorem \ref{thm:qualitative_thm}}. Indeed, since the biases are not scaled to zero as the widths diverge, their distribution has a non-trivial effect on the limiting infinite-width process. This is also the reason why the weights \(W^{(1)}\) need to be Gaussian. The same point is made in \cite{Hanin2023}.

Our proof is based on the Lindeberg replacement principle: we successively replace the weights of the network by Gaussian weights with matching first and second moments, controlling the error at each step. This reduces the problem to the case of Gaussian weights, where quantitative results from \cite{MR4898096} can be applied. We refer to \textbf{Section \ref{sec:Proof_sketch}} for a detailed discussion of the argument. This replacement is captured in the following weak comparison principle, which can be viewed as a Lindeberg principle for deep neural networks.

\begin{theorem}\label{thm:main-weak-bound}
    Let \(z^{(L+1)}\) and \(\tilde z^{(L+1)}\) be two DNNs in the sense of \textbf{Definition \ref{def:Deep_Neural_Network}}, with the same widths, with zero biases, and with an activation function \(\sigma\in C^{3\cdot2^{L-1}}_b\). Assume that the weights in \(z^{(L+1)}\) satisfy \textbf{Assumption \ref{assum:weights}} and \(\tilde z^{(L+1)}\) has Gaussian weights with matching first and second moments. Then
    \[
        \Bigl|\E[F(z^{(L+1)}(x))]-\E[F(\tilde z^{(L+1)}(x))]\Bigr|
        \leq C
        \sum_{k=1}^{L}\frac{1}{\sqrt{n_{k}}},
    \]
    for all \(F\in C^{3\cdot 2^{L-1}}_b(\R^{n_{L+1}})\).
    If the DNNs \(z^{(L+1)}\), \(\tilde z^{(L+1)}\) also have the same biases satisfying \textbf{Assumption \ref{assum:biases}}, and an activation function \(\sigma\in C^{3}_b(\R)\), then the same bound is true for all \(F\in L^\infty(\R^{n_{L+1}})\) with a  constant of the form \(C(L,n_{L+1,\sigma})\|F\|_{\infty}\). If instead, \(F\in C^1_b(\R^{n_{L+1}})\), the constant is of the form \(C(L,n_{L+1}, \sigma)\|\nabla F\|_{\infty}\).
\end{theorem}

Finally, \textbf{Theorem \ref{thm:main-weak-bound}} also yields, by a standard duality argument, a weaker but more direct bound on \(1-\)Wasserstein distance. When there are no biases the regularity requirements on the test function are so strong, that we will get a worse estimate than the one we can get from \textbf{Theorem \ref{thm:W_2_bound}}, since  $\Wcal_1\leq \Wcal_2$. When there are biases however, \textbf{Theorem \ref{thm:main-weak-bound}} yields a bound in the total variation distance and the \(1-\)Wasserstein distance with optimal rates in \(n_1,\ldots, n_L\):

\begin{corollary}
    Let \(z^{(L+1)}(x)\) be the output of a DNN with weights and biases satisfying \textbf{Assumptions \ref{assum:weights}} and \textbf{\ref{assum:biases}} and with activation function \(\sigma\in C^{3}_b(\R)\). Assuming that 
    \(K^{(\ell)}(x)\) is invertible for \(\ell=1,\ldots,n_{L+1}\), we have
    \begin{equation*}
        \max\{d_{TV}\bigl(z^{(L+1)}(x),\Zcal^{(L+1)}(x)\bigr),\,\Wcal_1\bigl(z^{(L+1)}(x),\Zcal^{(L+1)}(x)\bigr)\}\leq C\sum_{k=1}^{L}\frac{1}{\sqrt{n_{k}}},
    \end{equation*}
    where \(d_{TV}\) denotes the total variation distance.
\end{corollary}

\subsection{Frequently adopted notation}

We collect here some notation used throughout the paper.

\begin{itemize}
    \item For \(p\geq 1\), \(\Wcal_p(X,Y)\) denotes the \(p\)-Wasserstein distance between the laws of two random variables \(X\) and \(Y\):
    \[
        W_p(X,Y)
            :=
            \left(
            \inf_{\pi\in \Pi(X,Y)}
            \int_{\mathbb R^d\times \mathbb R^d}
            |x-y|^p\,d\pi(x,y)
            \right)^{1/p},
    \]
    where \(\Pi(X,Y)\) denotes the set of couplings of \(X\) and \(Y\).

    \item We write \(A\lesssim B\) if there exists a constant \(C>0\), independent of the widths \(n_1,\dots,n_L\), such that \(A\leq C B\).

    \item We denote the usual convolution between two functions $f$ and $g$ as $f \star g$.

    \item For a function \(F:\R^d\to\R\) and a multi-index of coordinates \(a_{1:r}=(a_1,\dots,a_r)\), we write
    \[
        \partial_{a_{1:r}}F=\partial_{a_1}\cdots\partial_{a_r}F \,,
    \]
    and we write $ |a_{1:r}| = r $.
    \item Given \(x\in\R^d\), the activation \(\sigma(x)\) is understood componentwise:
    \[
        \sigma(x)=(\sigma(x_1),\dots,\sigma(x_d)).
    \]
    \item For a function \(F:\R^d\to \R\) we denote
    \begin{equation}\label{eq:der_sum_quantities}
    \Lambda_r(F)=\sum_{a_1, \dots a_r = 1}^r\sup\limits_{x\in\R^d} \left|\partial_{a_{1:r}}F(x) \right| \,.
    \end{equation}
    \item We denote by $P_t=e^{t \Delta}$ the standard heat semigroup
    on $\R^d$; equivalently, one has $(P_t f)(x)=(p_t\star f)(x)$, where $p_t=(4\pi t)^{-d/2} \exp(-\sfrac{|x|^2}{4t})$, and $\widehat{P_t f}(\xi)=e^{-t|\xi|^2} \widehat f(\xi)$.
\end{itemize}

\section{Idea of the proof}\label{sec:Proof_sketch}
As mentioned in the introduction, our proof is based on the Lindeberg switching principle. The general result, in the context of the sums of i.i.d. random variables, is contained in the following result

\begin{proposition}\label{prop:Full_Lind_Prin}
    Let $(\bX_i)_{i\in\N}$ be a sequence of independent random vectors in $\R^d$. We assume that, for all $i\in\N$, the coefficients of $\bX_i$ are i.i.d. with zero mean, variance $\sigma_i^2$ and absolute third moments equal to $m_3^{(i)}$. Also, let $(\bY_i)_{i\in\N}$, be a sequence of independent Gaussian vectors in $\R^d$, so that the coefficients of $\bY_i$ are all independent and have matching first and second moments to the corresponding coefficients of $\bX_i$. Finally, let $F \in C^3_b(\R^d;\R)$. Then
    \[
        \biggl|\E\biggl[F\biggl(\sum_{i=1}^n\bX_i\biggr)\biggr]-\E\biggl[F\biggl(\sum_{i=1}^n\bY_i\biggr)\biggr]\biggr|\lesssim C_F\cdot d^3\cdot\sum_{i=1}^nm_3^{(i)},
    \]
    where the implied constant depends only on universal constants and 
    \[
        C_F:=\max \left\{||\partial^3_{i,j,k}F||_{L^{\infty}(\R^d)} \, : i,j,k=1,\dots d \right\}.
    \]
\end{proposition}

The proof is classical and is based on a Taylor expansion of \(F\) up to the third order. Due to the moment-matching condition on the random variables, only the third-order terms survive, which are then estimated crudely using the fact that \(F\in C^3_b(\R^d)\). 

This argument also produces the factor \(d^3\) on the right-hand side of the above bound, which points to the difficulty in applying this proposition to our case.  To see it more clearly, first consider the output \(z^{(L+1)}(x)\) with  weights \(W^{(\ell)}\) satisfying \textbf{Assumption \ref{eq:weight_assum}}.
\begin{definition}\label{def:switched_layers_DNN}
For \(K=-1,\ldots,L-1\), we define \(z^{(L+1;L-K)}(x)\) to be the output of a DNN with weights \(\tilde{W}^{(\ell)}\) defined as 
\[
    \tilde{W}^{(\ell)}=\begin{cases}
                        W^{(\ell)},\quad \text{if}\quad l=1,\ldots L-K,\\
                        G^{(\ell)},\quad \text{if}\quad l=L-K+1,\ldots,L+1,
                   \end{cases}
\]
where \(G^{(\ell)}\) is a \(n_{\ell}\times n_{\ell-1}\) random matrix with independent, mean-zero Gaussian coordinates, with second moment equal to \(C_W/n_{\ell-1}\). These are also independent from \((W^{(\ell)})_{\ell\in\{1,\ldots L+1\}}\). 
\end{definition}

In other words, \(z^{(L+1;L-K)}(x)\) is the output of the DNN \(z^{(L+1)}\), but where we switched the weights of the layers \(L-K+1,\ldots,L+1\)  to Gaussian weights with second moments matching the moments of the weights of the original DNN. 

To estimate some distributional distance of \(z^{(L+1)}(x)\) from its limit, we first can estimate 
\[
    \E[F(z^{(L+1)}(x))]-\E[F(\Zcal^{(L+1)}(x))],
\]
where \(\Zcal^{(L+1)}(x)\) is the Gaussian process as in \textbf{Theorem \ref{thm:qualitative_thm}}. Using the DNN with the switched layers \(z^{(L+1;L-K)}(x)\) we may write the above difference as
\begin{align}
     \E[F(z^{(L+1)}(x))]-\E[F(\Zcal^{(L+1)}(x))]=&
     \sum_{K=-1}^{L-2}\E[F(z^{(L+1;L-K)}(x))]-\E[F(z^{(L+1;L-K-1)}(x))]\nonumber
     \\&\nonumber
     +\E[F(z^{(L+1;1)}(x))]-\E[F(\Zcal^{(L+1)}(x))],
\end{align}
where we used the convention that \(z^{(L+1;L+1)}(x)=z^{(L+1)}(x)\). 
Observe that, by definition, the weights of \(z^{(L+1;1)}(x)\) are all Gaussian, and as such, we can bound the difference 
\[
    \E[F(z^{(L+1;1)}(x))]-\E[F(\Zcal^{(L+1)}(x))],
\]
using tools from Gaussian calculus. We are left now with estimating the differences
\[
    \bigg|\E[F(z^{(L+1;L-K)}(x))]-\E[F(z^{(L+1;L-K-1)}(x))]\bigg|,
\]
for \(K=1,\ldots,L-1\). When \(K=0\), we can condition on \(z^{(L)}(x)\)  and apply \textbf{Proposition \ref{prop:Full_Lind_Prin}} to get a bound of the form
\[
        |\E[F(z^{(L+1)}(x))]-\E[F(\tilde{z}^{(L+1;L)}(x))]|\lesssim C_F\cdot n_{L+1}^3\cdot\frac{1}{n_{L}^{3/2}}\E\biggl[\sum_{i=1}^{n_L}\sigma((z^{(L)}(x))_i)^3\biggr].
\]
A bound of this form also appears in \cite{Hanin2023} in Lemma $2.5$. Under our assumption, \(\sigma\) is bounded and as such we obtain
\[
    |\E[F(z^{(L+1)}(x))]-\E[F(\tilde{z}^{(L+1;L)}(x))]|\lesssim n_{L+1}^3\cdot\frac{1}{\sqrt {n_{L}}}.
\]
The problem arises when we want to apply \textbf{Proposition \ref{prop:Full_Lind_Prin}} to deeper layers. Indeed, for \(K>0\), by conditioning on \(z^{(L-K)}(x)\), and then applying \textbf{Proposition \ref{prop:Full_Lind_Prin}} we get a bound of the form

\begin{equation}\label{eq:naive_bound}
    \bigg|\E[F(z^{(L+1;L-K)}(x))]-\E[F(z^{(L+1;L-K-1)}(x))]\bigg|\lesssim n_{L-K+1}^3\cdot\frac{1}{\sqrt{n_{L-K}}},
\end{equation}
which goes to \(0\) only if \(n_{L-K}\gg n_{L-K+1}^{6}\). As such, to get a better bound, we must exploit the recursive structure of the network.

Let us look at the case $K=0$. We need to bound 
\[
    \E[F(z^{(L+1;L)}(x))]-\E[F(z^{(L+1;L-1)}(x))].
\]
Both networks have Gaussian weights at layer $L+1$, and they have matching weights up to layer $L-1$. Conditioning at the layer output $z^{(L-1)}$ yields 
\[
    \E \left[F(z^{(L+1;L)}(x)) \right]=\E\left[\E \left[F(G^{(L+1)}\sigma(W^{(L)}\sigma(z^{(L-1)}(x))))\bigg|z^{(L-1)}(x) \right]\right],
\]
and 
\[
     \E \left[F(z^{(L+1;L-1)}(x)) \right]=\E\left[\E \left[F(G^{(L+1)}\sigma(G^{(L)}\sigma(z^{(L-1)}(x)))\bigg|z^{(L-1)}(x) \right]\right],
\]
where we recall that the weights $G^{(L+1)}$, $G^{(L)}$ are Gaussian. To compare these last two quantities, we must apply the Lindeberg switching principle to 
\[
   W^{(L)}\sigma(z^{(L-1)}(x)),
\]
and replace $W^{(L)}$ by $G^{(L)}$ with a small error. The point is that, after conditioning on \(z^{(L-1)}(x)\) we are looking at a very specific function of the weights $W^{(L)}=(w^{(L)}_{i,j})_{i=1,\dots n_{L+1},j=1,\dots n_{L}}$, which is given by
\[
    \Fcal_L(x):=\E_{G^{(L+1)}}[F(G^{(L+1)}\sigma(x))],
\]
with \(x\in\R^{n_{L}}\).
It turns out that, when we follow the proof of the classical Lindeberg principle to this function, we can use this Gaussian averaging in the definition of \(\Fcal_{L+1}\) along with the scaling implicit in \(G^{(L+1)}\) to get an extra decay in \(n_{L}\) which will 'beat' the factor \(n_{L}^3\) in \eqref{eq:naive_bound} and will lead to a better bound than the generic one in \textbf{Proposition \ref{prop:Full_Lind_Prin}}.

The situation for \(K\geq1\) is a bit more complicated. Indeed, in that case, we can condition on  \(z^{(L-K)}(x)\), to get a function of the weights we want to 'switch' and can be defined recursively in \(K\) as
\[
    \Fcal_{L,K}(x)=\E[\Fcal_{L,K-1}(G^{(L-K+1)}\sigma(x))],
\]
with \(\Fcal_{L,-1}(x):=F(x)\). The issue here is that for \(K\geq1\), \(\Fcal_{L-K}\) is a function of \(n_{L-K}\) variables which will introduce an extra diverging factor in our estimates. Luckily, the same idea as before still works: using the fact that \(\Fcal_{L-K}\) is an appropriate average of \(F\) composed with Gaussian random variables, we will be able to get an extra decaying factor in \(n_{L-K}\) times higher derivatives of \(\Fcal_{L,K-1}\). We then continue recursively to get a bound that does not contain a diverging factor in the inner widths. We refer to \textbf{Theorem \ref{thm:DNN_Lindeberg_switching}} for more details. 

We point out that the recursion hinted at in the previous paragraph is the reason that we need a high regularity assumption on \(\sigma\). If, on the other hand, we include Gaussian biases, we can run this recursion much more effectively by estimating third derivatives of \(\Fcal_{L,K}\) by an extra decaying factor in \(n_{L-K}\) times third derivatives of \(\Fcal_{L,K-1}\). As such, we only need \(\sigma\in C^{3}_b(\R)\) in this case. For more details, look at \textbf{Theorem \ref{thm:improved_Lind_switching}}. 

The previous arguments can be used to prove \textbf{Theorem \ref{thm:main-weak-bound}}, which can be seen as a Lindeberg principle for fully connected deep neural networks. To get a bound for the Wasserstein distance, we will argue as follows. First, we make use of classical theorems bounding the \(W_2\) distance of a sum of i.i.d. random variables to a Gaussian random variable. In particular, we make use of \cite{1d_Wass_bound}. By conditioning on \(z^{(L)}(x)\) and applying the main result of \cite{1d_Wass_bound}, we get the bound
\[
    \Wcal_2(z^{(L+1)}(x),z^{(L+1;L)}(x))\lesssim \frac{1}{\sqrt{n_L}}.
\]
Therefore, we need to estimate \(\Wcal_2(\Zcal^{(L+1)}(x),z^{(L+1;L)}(x))\). We write
\[
    \Wcal_2(\Zcal^{(L+1)}(x),z^{(L+1;L)}(x))\leq\Wcal_2(z^{(L+1;1)}(x),z^{(L+1;L)}(x))+\Wcal_2(\Zcal^{(L+1)}(x),z^{(L+1;1)}(x)).
\]
For the second term, we will use the main result of \cite{MR4898096}, since \(z^{(L+1;1)}\) has Gaussian weights. For the first term, we observe that \(z^{(L+1;1)}(x)\) and \(z^{(L+1;L)}(x)\) are conditionally Gaussian random variables. As such, we can bound their \(\Wcal_2-\)distance using the following theorem, proved in \cite{entr_bounds_cond_gaussian}\footnote{For the sake of completeness, we reproduce the arguments of \cite{entr_bounds_cond_gaussian} proving Theorem \ref{th:entropic-bound} in Appendix \ref{app:Cond_Gauss_proof}.}.

\begin{theorem}\label{th:entropic-bound}
Let \(K\) be a deterministic symmetric positive definite \(d\times d\) matrix, and let
\(A\) be a random symmetric positive semidefinite \(d\times d\) matrix.
Let \(N\sim\mathcal N_d(0,I_d)\) be independent of \(A\), and define
\[
F:=\sqrt A\,N,
\qquad
G:=\sqrt K\,N.
\]
Then
\[
W_2(F,G)
\le
\frac{1}{\sqrt{\lambda(K)}}\,
\Bigl(\E\|A-K\|_{HS}^2\Bigr)^{1/2}
+
\frac{2^{3/2}d^{1/4}}{\lambda(K)^{3/2}}\,
\Bigl(\E\|A-K\|_{HS}^4\Bigr)^{1/2},
\]
where \(\lambda(K)\) is the smallest eigenvalue of \(K\).
\end{theorem}

It turns out the right-hand side of the above bound can be estimated using the bounds obtained in \textbf{Section \ref{sec:Lindeberg_bounds}}. This allows us to conclude.

\section{Lindeberg switching trick}\label{sec:Lindeberg_bounds}

In this section, we prove the Lindeberg principle in the context of DNNs, i.e., prove \textbf{Theorem \ref{thm:main-weak-bound}}. We split our main results into two cases: zero biases and Gaussian biases. 
In the case of zero biases, the main technical estimate is contained in the following Theorem:

\begin{theorem}\label{thm:DNN_Lindeberg_switching}
    Let \(z^{(L+1)}(x)\) be the output of a DNN with \( L\) hidden layers in the sense of \textbf{Definition} \ref{def:Deep_Neural_Network}. Assume that \( \sigma\in C^{3\cdot 2^{L-1}}_b(\R)\), where \(\sigma\) is the activation function, and that the weights are centered with finite third moment. Then for \(F\in C^{3\cdot 2^{L-1}}_b(\R^{n_{L+1}})\), we have
    \[
        \bigg|\E[F(z^{(L+1;L-K)}(x))]-\E[F(z^{(L+1;L-K-1)}(x))]\bigg|\lesssim \frac{1}{\sqrt{n_{L-K-1}}} { , \quad \forall K \in \{ -1, \dots, L-2\} \,,}
    \]
    where \(z^{(L;L-K)}(x)\) is as in \textbf{Definition \ref{def:switched_layers_DNN}}.
\end{theorem}

\begin{proof}
For \(K=-1\), we may use the usual Lindeberg exchange principle, \textbf{Proposition \ref{prop:Full_Lind_Prin}}. For \(K=0,\ldots, L-2\) we inductively define the following functions \(\Fcal_{L,K}:\R^{n_{L-K}}\to \R\)
\begin{equation}\label{eq:averaged_test_function}
    \Fcal_{L,K}(x)=\E[\Fcal_{L,K-1}(G^{(L-K+1)}\sigma(x))], \qquad K\geq0,
\end{equation}
with \(\Fcal_{L,-1}\equiv F\).

With this definition, we can write
\[
    \E[F(z^{(L+1;L-K)}(x))]=\E[\Fcal_{L,K}(W^{(L-K)}\sigma(z^{(L-K-1)}(x)))] \,,
\]
and
\[
    \E[F(z^{(L+1;L-K-1)}(x))]=\E[\Fcal_{L,K}(G^{(L-K)}\sigma(z^{(L-K-1)}(x))))] \,,
\]
If \(\sigma\in C^{3\cdot 2^{L-1}}_b(\R)\) and \(F\in C^{3\cdot 2^{L-1}}_b(\R^{n_{L+1}})\) then \(\Fcal_{L,K}\in C^{3\cdot 2^{L-1}}_b(\R^{n_{L-K}})\). To estimate 
\begin{align}\label{eq:basic_Lind_comparison}
    \bigg|\E[F(z^{(L+1;L-K)}(x))]-\E[F(z^{(L+1;L-K-1)}(x))]\bigg|=
\end{align}
\[
    \bigg|\E[\Fcal_{L,K}(W^{(L-K)}\sigma(z^{(L-K-1)}(x)))]-\E[\Fcal_{L,K}(G^{(L-K)}\sigma(z^{(L-K-1)}(x)))]\bigg| \,,
\]
we will estimate 
\begin{equation}\label{eq:aux_Lind_diff}
    \E[\Fcal_{L,K}(W^{(L-K)}\sigma(y))]-\E[\Fcal_{L,K}(G^{(L-K)}\sigma(y))] \,
\end{equation}
uniformly in  \(y\in \R^{n_{L-K-1}}\). The latter follows the basic steps of the proof of the classical Lindeberg switching principle closely, see for example \cite{zygouras_discrete_stochanal_notes}.

We define the random vectors \(\bX_i\), \(\bY_i\in\R^{n_{L-K}}\), \(i=1,\ldots,n_{L-K-1}\) with components
\begin{equation}
\label{eq:def_X_coordinates}
    \bX_i(j)=\frac{1}{\sqrt{n_{L-K-1}}}w_{i,j}^{(L-K)}\sigma(y_j),\quad\text{and}\quad \bY_i(j)=\frac{1}{\sqrt{n_{L-K-1}}}g_{i,j}^{(L-K)}\sigma(y_j).
\end{equation}
With this notation we can rewrite \eqref{eq:aux_Lind_diff} as 
 \begin{equation}\label{eq:diff_rewrite}
        \biggl|\E\biggl[\Fcal_{L,K}\biggl(\sum_{i=1}^{n_{L-K-1}}\bX_i\biggr)\biggr]-\E\biggl[\Fcal_{L,K}\biggl(\sum_{i=1}^{n_{L-K-1}}\bY_i\biggr)\biggr]\biggr|\,,
 \end{equation}
To estimate this, we define
\begin{equation}\label{eq:sum_for_tel}
        \bS_{n_{L-K-1};K}:=\sum_{i=1}^{n_{L-K-1}-m}\bX_i+\bY_{n_{L-K-1}-m+2}+\bY_{n_{L-K-1}-m+3}+\dots+\bY_{n_{L-K-1}}.
\end{equation}
Since,
\[
    \biggl|\E\biggl[\Fcal_{L,K}\biggl(\sum_{i=1}^{n_{L-K-1}}\bX_i\biggr)\biggr]-\E\biggl[\Fcal_{L,K}\biggl(\sum_{i=1}^{n_{L-K-1}}\bY_i\biggr)\biggr]\biggr|\leq
\]
\[
    \sum_{m=1}^{n_{L-K-1}}|\E[\Fcal_{L,K}(\bS_{n_{L-K-1};k}+\bX_{n_{L-K-1}-m+1})]-\E[\Fcal_{L,K}(\bS_{n_{L-K-1};m}+Y_{n_{L-K-1}-m+1})]| \,,
\]
it is sufficient to bound the distance
\begin{equation}\label{eq:basic_diff}
        |\E[\Fcal_{L,K}(\bS_{n_{L-K-1};m}+\bX_{n_{L-K-1}-m+1})]-\E[\Fcal_{L,K}(\bS_{n_{L-K-1};m}+Y_{n_{L-K-1}-m+1})]| \,.
\end{equation}
From Taylor, we get
\[
    \Fcal_{L,K}(\mathbf{x}+\mathbf{h})=\Fcal_{L,K}(\mathbf{x})+\sum_{a=1}^{n_{L-K}}\partial_a\Fcal_{L,K}(\mathbf{x})h_a+\frac{1}{2}\sum_{a,b=1}^{n_{L-K}}\partial^2_{a,b}\Fcal_{L,K}(\mathbf{x})h_ah_b+R(\mathbf{x};\mathbf{h}) \,.
\]
where
\[
R(\mathbf{x};\mathbf{h})=\sum_{a,b,c=1}^{n_{L-K}}\Rcal_{a,b,c}(\mathbf{x})h_ah_bh_c \,,
\]
with 
\[
    \Rcal_{a,b,c}(\mathbf{x})=\frac{1}{2}\int_0^1(1-t)^2\partial^3_{a,b,c}\Fcal_{L,K}(\mathbf{x}+t\mathbf{h})dt \,.
\]
Now, since the components of \(\bX_{n_{L-K-1-m+1}}\) and \(\bY_{n_{L-K-1-m+1}}\) have matching first and second moments, \eqref{eq:basic_diff} is bounded above by
\[
    \E[|R(\bS_{n_{L-K-1};k};\bX_{n_{L-K-1}-m+1}))|]+\E[|R(\bS_{n_{L-K-1};k};\bY_{n_{L-K-1}-m+1}))|].
\]
Assuming that 
\begin{equation}\label{eq:technical_estimate}
    \sum_{a,b,c=1}^{n_{L-K}}\sup_{x\in\R^{n}}|\partial_{a,b,c}\Fcal_{L,K}(x)|\lesssim 1,
\end{equation}
it is straightforward to estimate these terms. For example recalling \eqref{eq:def_X_coordinates}
\begin{align*}
    \E[|R(\bS_{n_{L-K-1};m},\bX_{n_{L-K-1}-m+1})|] & \lesssim 
   \sum_{a,b,c=1}^{n_{L-K}}\sup_{x\in\R^{n}}|\partial_{a,b,c}\Fcal_{L,K}(x)| \\
   & \cdot \E \left[ \left|\bX_{n_{L-K-1}-m+1}(a)\bX_{n_{L-K-1}-m+1}(b)
   \cdot\bX_{n_{L-K-1}-m+1}(c) \right| \right] 
   \\
   & \lesssim \frac{1}{n_{L-K-1}^{3/2}} \,,
\end{align*}
where we used H\"older's inequality and the fact that \(\sigma\) is bounded.
Similarly, using the identity
\[
    \E[|Y|^p]=\E[|Y|^2]^{p/2}2^{p/2}\frac{\Gamma(\frac{p+1}{2})}{\sqrt{\pi}}.
\]
where \(Y\) is a Gaussian mean zero random variable with variance \(1\), we get the same bound for $\E[|R(\bS_{n_{L-K-1};m},\bX_{n_{L-K-1}-m+1})|]$, which implies that \eqref{eq:basic_diff} is \(O(1/n_{L-K-1}^{3/2})\). Summing over \(m\) implies that \eqref{eq:aux_Lind_diff} is \(O(1/\sqrt{n_{L-K-1}})\), uniformly in \(y\). This proves that \eqref{eq:basic_Lind_comparison} is \(O(1/\sqrt{n_{L-K-1}})\).

Therefore, we are left to show \eqref{eq:technical_estimate}. Recalling \eqref{eq:der_sum_quantities}, it is sufficient to show that 
\begin{equation}\label{eq:der_estimate}
    \Lambda_3(\Fcal_{L,K})\lesssim 1.
\end{equation}
We do this via a recursion on \(K\). We will show
\begin{equation}\label{eq:basic_rec}
    \Lambda_{r}(\Fcal_{L,K})\lesssim \sum_{j=2}^{2r}\Lambda_j(\Fcal_{L,K-1}),
\end{equation}
for all \(r\in \N\), and \(\sigma\in C^{r}_b(\R)\). Since \(\Fcal_{L,-1}=F\), we have \(\Lambda_r(F)\lesssim 1\), for all \(r=1,\ldots,3\cdot2^{L-1}\), we can conclude.

Recalling the recursive definition  in \eqref{eq:averaged_test_function}, we can write 
 \[
        \Fcal_{L,K}(x)= P_{q_K(x)/2}\Fcal_{L,K-1}(0)=: \Phi(q_K(x)/2),
    \]
    where \(P_t= e^{t\Delta}\) and \(\Phi(t):=P_t\Fcal_{L,K-1}(0)\) and 
    \begin{equation}\label{eq:q_definition}
        q_K(x):=\frac{C_W}{n_{L-K}}\sum_{i=1}^{n_{L-K}}\sigma(x_i)^2.
    \end{equation}
     Indeed, we have \(\Fcal_{L,K}(x)=\E[\Fcal_{L,K-1}(\sqrt{q_K(x)}Z)]\), where \(Z\) is an \( n_{L-K+1}-\)dimensional vector with independent Gaussian entries of mean zero and variance \(1\).
     
    Now, we observe that
    \begin{equation}\label{eq:Phi_deriv}
        \frac{d^k}{d^kt}\Phi(t)= \E[\Delta^k\Fcal_{L,K-1}(\sqrt{t}Z)].
    \end{equation}
    Now, we have
    \begin{equation}\label{eq:Faa_di_Bruno}
        \partial_{a_{1:r}}\Fcal_{L,K}(x)=\sum_{k=1}^r\frac{d^k}{d^kt}\Phi(q(x))\sum_{\pi\in \Pi_k}\prod_{B\in\pi}\partial_{a_{B}}q(x),
    \end{equation}
    where we used the Fa\'a di Bruno formula \cite{MR507345}. We recall that \(\Pi_k\) denotes the set of partitions of \(\{1,\ldots,r\}\) into \(k\) parts,  and \(\partial_{a_B}:=\prod_{i\in B}\partial_{a_i}\). Recalling \eqref{eq:q_definition}, we see that, for \(B=\{i_1,\dots,i_{|B|}\}\)
    \[
        \partial_{a_B}q(x)=\frac{C_W}{n_{L-K}}\frac{d^{|B|}}{dx^{|B|}}(\sigma(x_{a_{i_1}})^2)\mathbb{1}_{\{a_{i_1}=a_{i_2}\ldots=a_{i_{|B|}}\}}.
    \]
    since \(\sigma\in C_b^{2r}(\R)\), this identity implies that
    \begin{equation}\label{eq:combinatorial_est}
        \biggl|\sum_{a_{1:r}=1}^{n_{L-K}}\sum_{\pi\in \Pi_k}\prod_{B\in\pi}\partial_{a_{B}}q(x)\biggr|\lesssim1.
    \end{equation}
    Now, by plugging \eqref{eq:Phi_deriv} to \eqref{eq:Faa_di_Bruno}, summing over \(a_{1:r}\), and using the above estimate we have
    \[
        \Lambda_r(\Fcal_{L,K})\lesssim \sum_{k=1}^r\sup_{x}\E[|\Delta^k \Fcal_{L,K-1}(\sqrt {q(x)}Z)]\lesssim \sum_{k=1}^{2r}\Lambda_{k}(\Fcal_{L,K-1}),
    \]
    which concludes the proof.

\end{proof}

When the DNN has additional Gaussian biases in each layer, we can avoid this recursion by leveraging the extra smoothing afforded by the biases. In particular, we have the following theorem.

\begin{theorem}\label{thm:improved_Lind_switching}
    Let \(z^{(L+1)}\) be the output of a DNN with \(L\) hidden layers with Gaussian bias as in \eqref{eq:Gaussian_biases} and with activation function \(\sigma\in C_b^3(\R)\). Then
    \[
        \bigg|\E[F(z^{(L+1;L-K)}(x))]-\E[F(z^{(L+1;L-K-1)}(x))]\bigg|\lesssim \|F\|_{\infty}\frac{1}{\sqrt{n_{L-K-1}}} { , \quad \forall K \in \{ -1, \dots, L-2\} \,,}
    \]
    if \(F\in L^\infty(\R^{n_{L+1}})\). If \(F\in C^1_b(\R^{n_{L+1}})\) the same bound holds with \(\|\nabla F\|_{\infty}\) instead of \(\|F\|_\infty\).
\end{theorem}

\begin{proof}
    We follow the same steps as in the proof of \textbf{Theorem \ref{thm:DNN_Lindeberg_switching}}. First, we define 'biased' versions of the functions \(\Fcal_{L,K}\), defined in \eqref{eq:averaged_test_function}:
    \begin{equation}\label{eq:biased_averaged_function}
        \Fcal_{L,K}^{\rm biased}(x):=\E[\Fcal_{L,K-1}^{\rm biased}(G^{(L-K+1)}\sigma(x)+b^{(L-K+1)})], \qquad K\geq0,
    \end{equation}
    with \(\Fcal_{L,-1}^{\rm biased}(x) \equiv F\). We also define 
    \begin{equation}\label{eq:double_averaged_test_function}
        \widetilde\Fcal_{L,K}^{\rm biased}(x) := \E [\Fcal^{\rm biased}_{L,K}(x+b^{(L-K)})]= P_{c_b^2/2}\Fcal^{\rm biased}_{L,K}(x).
    \end{equation}
   Observe that \(\widetilde \Fcal^{\rm biased}_{L,K}\) is a smooth function for any measurable activation function \(\sigma\) and observable \(F\). Arguing similarly to the proof of \textbf{Theorem \ref{thm:DNN_Lindeberg_switching}}, it suffices to show
   \[
        \bigl|\E[\widetilde\Fcal^{\rm biased}_{L,K}(W^{(L-K)}\sigma(y))]-\E[\widetilde\Fcal_{L,K}^{\rm biased}(G^{(L-K)}\sigma(y))]\bigr|\lesssim \frac{1}{\sqrt{n_{L-K-1}}},
   \]
    which will follow, using the same arguments, from the following estimate
    
    \begin{equation}\label{eq:biased_technical_estimate}
        \sum_{a,b,c=1}^{n_{L-K}}\sup_{x\in\R^{n}}|\partial_{a,b,c}\widetilde\Fcal^{\rm biased}_{L,K}(x)|\lesssim \begin{cases}
            \|F\|_{\infty}, \quad\textit{if }\; F\in L^\infty(\R^{n_{L+1}})\\
            \\
            \|\nabla F\|_{\infty},\quad \textit{if }\; F\in C^1_b(\R^{n_{L+1}}),
            
        \end{cases}
    \end{equation}
   which is the 'biased' version of the bound \eqref{eq:technical_estimate}. 
   To show \eqref{eq:biased_technical_estimate}, we want to follow the same steps as in the proof of \eqref{eq:technical_estimate} in the previous proof, but in a way that avoids estimating quantities like \(\Lambda_r(\Fcal_{L,K})\), \(r=1,2,3\) by above by \(\Lambda_{r'}(\Fcal_{L,K})\) with \(r'> r\)  which requires more derivatives in \(\sigma\).
   
   To do this, we write
    \[
        \Fcal_{L,K}^{\rm biased}(x)= \phi_{L,K}(q_{K}(x)),
    \]
    where \(\phi_{L,K}(s)=P_{s+c_b^2/2}\Fcal^{\rm biased}_{L,K-1}(0)\), and \(q_{K}(x)\) is as in \eqref{eq:q_definition}. Using \eqref{eq:biased_averaged_function}, it is easy to see that \(\phi_{L,K}\) satisfy the following recursion 
    \begin{equation}\label{eq:phi_recursion}
        \phi_{L,K}(s) = \E\bigg[\phi_{L,K-1}\bigg(\frac{C_W}{n_{L-K+1}}\sum_{i=1}^{n_{L-K+1}}g_s(U_i)\bigg)\bigg],\quad U_i\sim\Ncal(0,1),\, \textit{\rm independent},
    \end{equation}
    where 
    \begin{equation}
        g_s(u):= \sigma\left(\left(\sqrt{s+c_b^2}\right)u \right)^2,
    \end{equation}
    with \(c_b\) the variance of the bias \(b^{(L-K+1)}\), specified in \eqref{eq:Gaussian_biases}. Using this recursion, we will prove that
    \begin{equation}\label{eq:biased_der_estimate}
        \Lambda_r(\phi_{L,K})=\sup_{s\ge 0}\bigl|\phi_{L,K}^{(r)}(s)\bigr|\lesssim\begin{cases}
            \|F\|_{\infty}, \quad\textit{if }\; F\in L^\infty(\R^{n_{L+1}})\\
            \\
            \|\nabla F\|_{\infty},\quad \textit{if }\; F\in C^1_b(\R^{n_{L+1}})
            
        \end{cases}
        \qquad r=1,2,3,
    \end{equation}
    for all \(K\geq1\). With this bound, \eqref{eq:biased_technical_estimate} follows by arguing similarly to the proof of \eqref{eq:der_estimate}. Indeed, we have
    \[
        \partial_{a_{1:3}}\Fcal_{L,K}^{\rm biased}(x)=\sum_{r=1}^3\frac{d^r}{d^rt}\phi_{L,K}(q_K(x))\sum_{\pi\in \Pi_3}\prod_{B\in\pi}\partial_{a_{B}}q_K(x),
    \]
    where we once again used the Fa\'a di Bruno formula. Now, from \eqref{eq:combinatorial_est} and \eqref{eq:biased_der_estimate} we get 
    \[
        \sum_{a,b,c=1}^{n_{L-K}}\sup_{x\in\R^{n}}|\partial_{a,b,c}\Fcal^{\rm biased}_{L,K}(x)|\lesssim \begin{cases}
            \|F\|_{\infty}, \quad\textit{if }\; F\in L^\infty(\R^{n_{L+1}})\\
            \\
            \|\nabla F\|_{\infty},\quad \textit{if }\; F\in C^1_b(\R^{n_{L+1}}).
            
        \end{cases}
    \]
    Since \(\tilde\Fcal_{L,K}(x)\) is given by \eqref{eq:double_averaged_test_function}, this proves \eqref{eq:biased_technical_estimate}.
    
    As mentioned, to prove \eqref{eq:biased_der_estimate}, we will make use of the recursion \eqref{eq:phi_recursion}. First, we record some preliminary estimates: A straightforward calculation implies that, since \(c_b>0\) and \(\sigma\in C^{3}_b(\R)\)
    \begin{equation}\label{eq:der_g_s_bounds}
        |\partial_s g_s(u)|\lesssim 1+|u|,
        \qquad
        |\partial_s^2 g_s(u)|\lesssim 1+|u|^2,
        \qquad
        |\partial_s^3 g_s(u)|\lesssim 1+|u|^3,
    \end{equation}
    for all $s\ge 0$ and $u\in\mathbb R$.
    Hence, if $U\sim N(0,1)$, then for every fixed $p\ge 1$,
    \begin{equation}\label{eq:averaged_g_s_der_bounds}
        \sup_{s\ge 0}\mathbb E|\partial_s g_s(U)|^p<\infty,
        \qquad
        \sup_{s\ge 0}\mathbb E|\partial_s^2 g_s(U)|^p<\infty,
        \qquad
        \sup_{s\ge 0}\mathbb E|\partial_s^3 g_s(U)|^p<\infty.
    \end{equation}
    Using the bounds \eqref{eq:averaged_g_s_der_bounds} and Jensen's inequality, we get for every fixed $p\ge 1$,
    \begin{equation}\label{eq:A_s_bounds}
        \sup_{s\ge 0,}\mathbb E\bigg|\frac{1}{n_{L-K+1}}\sum_{i=1}^{n_{L-K+1}} \partial_s^rg_s(U_i)\bigg|^p<\infty,
    \end{equation}
    for \(r=1,2,3\).
    Using these bounds, the chain rule and the recursion \eqref{eq:phi_recursion}, it is easy to see that for \(\sigma\in C_b^3(\R)\)
    \begin{equation}\label{eq:M_r_recursion}
       \sum_{r=1}^3 \Lambda_r(\phi_{L,K})\lesssim \sum_{r=1}^3 \Lambda_r(\phi_{L,K-1}),
    \end{equation}
    Finally, observe that we can use the following bounds 
    \[
        \Lambda_r(\phi_{L,0})\lesssim\begin{cases}
            \|F\|_{\infty}, \quad\textit{if }\; F\in L^\infty(\R^{n_{L+1}})\\
            \\
            \|\nabla F\|_{\infty},\quad \textit{if }\; F\in C^1_b(\R^{n_{L+1}}),
            
        \end{cases}
    \]
    for \(r=1,2,3\).  This proves \eqref{eq:biased_der_estimate} and concludes the proof.
\end{proof}

\begin{remark}\label{rem:Applicability}
    The previous theorems compare the output of a DNN with Gaussian weights in the last \(k\) layers to the output of a DNN with Gaussian weights in the last \(k+1\) layers. For our purposes, we will also need to estimate
    \[
        \bigg|\E[F(z^{(L;L-K)}(x))]-\E[F(z^{(L;L-K-1)}(x))]\bigg|.
    \]
    As usual the problem is that \( z^{(L;L-K)}(x)\) is an \(n_L-\)dimensional vector, and that \(n_L\to\infty\). It is clear that in this case, there is no hope of estimating the above quantity uniformly in \(n_L\), for general \( F\). Luckily, there is no need. As an application of \textbf{Theorem \ref{th:entropic-bound}} below, we will only need to consider functions \(F\) that depend on a fixed number of components of \(z^{(L;L-K)}(x)\). For such functions, we can apply \textbf{Theorem \ref{thm:DNN_Lindeberg_switching}} to obtain
    \begin{equation}\label{eq:second_to_last_Lind_switching}
        \bigg|\E[F(z^{(L;L-K)}_{i_1}(x),\ldots,z^{(L;L-K)}_{i_\alpha}(x))]-\E[F(z^{(L;L-K-1)}_{i_1}(x),\ldots, z^{(L;L-K-1)}_{i_\alpha}(x))]\bigg|\leq C\frac{1}{\sqrt{n_{L-K-1}}},
    \end{equation}
    where \(i_1,\dots i_\alpha\) are some fixed indices and \(C= C(L,\alpha,\sigma)\) some constant.
\end{remark}

Now, \textbf{Theorem \ref{thm:main-weak-bound}} can be proved by a simple telescoping sum trick, as described in \textbf{Section \ref{sec:Proof_sketch}}:

\begin{proof}[Proof of \textbf{Theorem \ref{thm:main-weak-bound}}]
    We can write
    \begin{align*}
        \E[F(z^{(L+1)}(x))]-\E[F(\Zcal^{(L+1)}(x))]=&
     \sum_{K=-1}^{L-2}\E[F(z^{(L+1;L-K)}(x))]-\E[F(z^{(L+1;L-K-1)}(x))]\nonumber
     \\&
     +\E[F(z^{(L+1;1)}(x))]-\E[F(\Zcal^{(L+1)}(x))].
    \end{align*}
    Applying \textbf{Theorem \ref{thm:DNN_Lindeberg_switching}} in the case of zero biases or \textbf{Theorem \ref{thm:improved_Lind_switching}} in the case of Gaussian biases, yields that the sum in the above equation is bounded above by \(C\sum_{K=0}^{L-1}\frac{1}{\sqrt{n_{L-K}}}\). From \cite[Theorem \(3.3\)]{MR4898096}, we have
    \[
        \E[F(z^{(L+1;1)}(x))]-\E[F(\Zcal^{(L+1)}(x))]\lesssim \sum_{K=0}^{L-1}\frac{1}{\sqrt{n_{L-K}}},
    \]
    which concludes the proof.
\end{proof}

\section{Proof of \textbf{Theorem \ref{thm:W_2_bound}} and \textbf{Theorem \ref{thm:Improved_W_2_bound}}}\label{sec:W_2_bounds_proofs}
In this section, we prove \textbf{Theorem \ref{thm:W_2_bound}}. The proof of \textbf{Theorem \ref{thm:Improved_W_2_bound}} is similar, only we use \textbf{Theorem \ref{thm:improved_Lind_switching}} instead of \textbf{Theorem \ref{thm:DNN_Lindeberg_switching}}. The rest of the steps are the same.

We follow the sketch described in \textbf{Section \ref{sec:Proof_sketch}}. As a first step, we want to replace the weights in the last layer by mean-zero Gaussian weights with matching variances, with a small \(\Wcal_2\)-error. First, we need the following lemma

\begin{lemma}\label{lem:conditional-w2}
Let \(d\in\N\), let \(X,Y\) be square-integrable \(\R^d\)-valued random variables, and let \(\Fcal\) be a sub-\(\sigma\)-field. Then
\[
\Wcal_2(X,Y)^2
\le
\E\Bigl[
\Wcal_2\bigl(\mathcal L(X\mid\Fcal),\mathcal L(Y\mid\Fcal)\bigr)^2
\Bigr].
\]
\end{lemma}

\begin{proof}
For \(\omega\in\Omega\), write
\[
\mu_\omega:=\mathcal L(X\mid\Fcal)(\omega),
\qquad
\nu_\omega:=\mathcal L(Y\mid\Fcal)(\omega).
\]
Choose a probability kernel
\(\omega\mapsto \pi_\omega\) such that, for a.e.\ \(\omega\), \(\pi_\omega\) is an
optimal coupling of \(\mu_\omega\) and \(\nu_\omega\). By measurable selection \cite[Corollary \(5.22\)]{Villani2009}, we may choose \(\omega\to \pi_\omega\) to be measurable. As such, we can define a probability
measure \(\pi\) on \(\R^d\times\R^d\) by
\[
\pi(A):=\E[\pi_\omega(A)],
\qquad A\in\mathcal B(\R^d\times\R^d).
\]
Then \(\pi\) is a coupling of \(\mathcal L(X)\) and \(\mathcal L(Y)\). Hence
\[
\Wcal_2(X,Y)^2
\le
\int_{\R^d\times\R^d} |x-y|^2\,d\pi(x,y)
=
\E\Bigl[
\int_{\R^d\times\R^d} |x-y|^2\,d\pi_\omega(x,y)
\Bigr].
\]
Since \(\pi_\omega\) is optimal,
\[
\int_{\R^d\times\R^d} |x-y|^2\,d\pi_\omega(x,y)
=
\Wcal_2(\mu_\omega,\nu_\omega)^2
\]
for a.e.\ \(\omega\). This gives the claim.
\end{proof}

\begin{lemma}\label{lem:last_layer_replacement}
If \(z^{(L+1)}(x)\) is the output of a DNN with zero biases, weights satisfying \textbf{Assumption \ref{assum:weights}} and a bounded activation function, we have
\[
\Wcal_2\bigl(z^{(L+1)}(x), z^{(L+1;L)}(x)\bigr)
\lesssim
\sqrt{\frac{n_{L+1}}{n_L}},
\]
where \(z^{(L+1;L)}(x)\) is as in \textbf{Definition \ref{def:switched_layers_DNN}}. The same bound is true if \(z^{(L+1)}(x)\) has Gaussian biases, according to \textbf{Assumption \ref{assum:biases}}.
\end{lemma}

\begin{proof}
We only prove the case of zero biases, as the case of Gaussian biases is identical.

First, recall a one-dimensional Wasserstein estimate, obtained by specializing a theorem of Bonis \cite[Theorem 1]{1d_Wass_bound}.
If \(S_n=\sum_{i=1}^n X_i\) is a sum of independent centered real-valued random variables and \(G\) is Gaussian with the same variance as \(S_n\), then
\begin{equation}\label{eq:1d_wass_bound}
\Wcal_2(S_n,G)^2\lesssim \frac{\sum_{i=1}^n \E[|X_i|^4]}{\sum_{i=1}^n \E[X_i^2]}.
\end{equation}
The implicit constant depends only on the universal constant in \cite{1d_Wass_bound}.

Let $\Fcal_L$ be the $\sigma$-algebra generated by $z^{(L)}(x)$, and set
\[
a_i:=\sqrt{\frac{C_W}{n_L}}\,\sigma\bigl(z_i^{(L)}(x)\bigr),
\qquad i=1,\dots,n_L.
\]
Then, conditionally on \(\Fcal_L\), the coefficients \((a_i)_{i=1}^{n_L}\) are
deterministic.

For each \(r=1,\dots,n_{L+1}\), write
\[
z_r^{(L+1)}(x)=\sum_{i=1}^{n_L} a_i\,w^{(L+1)}_{r,i} \,,
\qquad
z_r^{(L+1;L)}(x)=\sum_{i=1}^{n_L} a_i\,g^{(L+1)}_{r,i} \,,
\]
where \((w^{(L+1)}_{r,i})_{r,i}\) are i.i.d.\ with law \(\mu\) specified in \textbf{Assumption \ref{assum:weights}}, and
\((g^{(L+1)}_{r,i})_{r,i}\) are i.i.d.\ standard Gaussian random variables, all
independent of \(\Fcal_L\).

Set
\[
v:=\sum_{i=1}^{n_L} a_i^2
=
\frac{C_W}{n_L}\sum_{i=1}^{n_L}\sigma\bigl(z_i^{(L)}(x)\bigr)^2.
\]
If \(v=0\), then \(a_i=0\) for all \(i\), and therefore
\[
z^{(L+1)}(x)=z^{(L+1;L)}(x)=0
\]
almost surely conditionally on \(\Fcal_L\). Thus the conditional
\(\Wcal_2\)-distance is zero.

Assume now that \(v>0\). Fix \(r\in\{1,\dots,n_{L+1}\}\), and define
\[
X_{r,i}:=a_i\,w^{(L+1)}_{r,i},
\qquad i=1,\dots,n_L.
\]
Conditionally on \(\Fcal_L\), the variables \(X_{r,i}\) are independent,
centered, and satisfy
\[
\sum_{i=1}^{n_L}\E[X_{r,i}^2\mid\Fcal_L]=v.
\]
Moreover,
\[
z^{(L+1)}_r(x)=\sum_{i=1}^{n_L} X_{r,i}.
\]
Let \(G\sim\mathcal N(0,1)\), independent of everything else. Applying \eqref{eq:1d_wass_bound} (conditionally on \(\Fcal_L\)), we obtain
\[
\Wcal_2\bigl(\mathcal L(z^{(L+1)}_r(x)\mid\Fcal_L),\mathcal L(\sqrt v\,G\mid\Fcal_L)\bigr)^2
\lesssim
v^{-1}\sum_{i=1}^{n_L}\E[|X_{r,i}|^4\mid\Fcal_L]
\lesssim
\frac{\sum_{i=1}^{n_L} a_i^4}{\sum_{i=1}^{n_L} a_i^2}.
\]
Since \(z_r^{(L+1;L)}(x)\mid\Fcal_L\sim\mathcal N(0,v)\), this gives
\[
\Wcal_2\bigl(\mathcal L(z^{(L+1)}_r(x)\mid\Fcal_L),\mathcal L(z^{(L+1;L)}_r(x)\mid\Fcal_L)\bigr)^2
\lesssim
\frac{C_W}{n_L}
\frac{\sum_{i=1}^{n_L}\sigma\bigl(z_i^{(L)}(x)\bigr)^4}
     {\sum_{i=1}^{n_L}\sigma\bigl(z_i^{(L)}(x)\bigr)^2}
\,\one_{\{\sum_i \sigma(z_i^{(L)}(x))^2>0\}}.
\]

Now observe that, conditionally on \(\Fcal_L\), the coordinates
\((z_1^{(L+1)}(x),\dots,z_{n_{L+1}}^{(L+1)}(x))\) are independent, and the coordinates \((z_1^{(L+1;L)}(x),\dots,z_{n_{L+1}}^{(L+1;L)}(x))\) are
also independent. Hence
\[
\mathcal L\bigl(z^{(L+1)}(x)\mid\Fcal_L\bigr)
=
\bigotimes_{r=1}^{n_{L+1}} \mathcal L(z_r^{(L+1)}(x)\mid\Fcal_L),
\qquad
\mathcal L\bigl(z^{(L+1;L)}(x)\mid\Fcal_L\bigr)
=
\bigotimes_{r=1}^{n_{L+1}} \mathcal L(z_{r}^{(L+1;L)}(x)\mid\Fcal_L).
\]
Taking the product of optimal couplings in each coordinate, we get
\[
\Wcal_2\bigl(\mathcal L(z^{(L+1)}(x)\mid\Fcal_L),
             \mathcal L(z^{(L+1;L)}(x)\mid\Fcal_L)\bigr)^2
\le
\sum_{r=1}^{n_{L+1}}
\Wcal_2\bigl(\mathcal L(z_r^{(L+1)}(x)\mid\Fcal_L),\mathcal L(z_r^{(L+1;L)}(x)\mid\Fcal_L)\bigr)^2.
\]
Therefore
\[
\Wcal_2\bigl(\mathcal L(z^{(L+1)}(x)\mid\Fcal_L),
             \mathcal L(z^{(L+1;L)}(x)\mid\Fcal_L)\bigr)^2
\lesssim
\frac{n_{L+1}\,C_W}{n_L}
\frac{\sum_{i=1}^{n_L}\sigma\bigl(z_i^{(L)}(x)\bigr)^4}
     {\sum_{i=1}^{n_L}\sigma\bigl(z_i^{(L)}(x)\bigr)^2}
\,\one_{\{\sum_i \sigma(z_i^{(L)}(x))^2>0\}}.
\]
Since \(\sigma\) is bounded, on the event where the denominator is positive,
\[
\frac{\sum_i \sigma(z_i^{(L)}(x))^4}
     {\sum_i \sigma(z_i^{(L)}(x))^2}
\le
\|\sigma\|_\infty^2.
\]
Thus
\[
\Wcal_2\bigl(\mathcal L(z^{(L+1)}(x)\mid\Fcal_L),
             \mathcal L(z^{(L+1;L)}(x)\mid\Fcal_L)\bigr)^2
\lesssim
\frac{n_{L+1}}{n_{L}}.
\]
Finally, \textbf{Lemma \ref{lem:conditional-w2}} yields
\[
\Wcal_2\bigl(z^{(L+1)}(x), z^{(L+1;L)}(x)\bigr)^2
\lesssim
\E\Bigl[
\Wcal_2\bigl(\mathcal L(z^{(L+1)}(x)\mid\Fcal_L),
             \mathcal L(z^{(L+1;L)}(x)\mid\Fcal_L)\bigr)^2
\Bigr]
\lesssim
\frac{n_{L+1}}{n_L}.
\]
This proves the claim.
\end{proof}

Before moving on with the proof of \textbf{Theorem \ref{thm:W_2_bound}}, we need an additional lemma, which quantifies the rate of convergence of a DNN with Gaussian weights to its infinite width limit. We will only need this for special observables of the layer \(L\). In particular, we have the following lemma

\begin{lemma}\label{lem:gaussian-hidden-layer-to-gp}
Fix \(\alpha\in\{1,2,3,4\}\) and \(x\in\R^{n_0}\).
Assume that the infinite-width covariances are non-degenerate on the singleton input set \(\{x\}\) in the sense of \cite{trevisan2023widedeepneuralnetworks}.
Let
\[
\Phi_\alpha(u_1,\dots,u_\alpha):=\prod_{r=1}^\alpha \sigma(u_r)^2.
\]
Then there exists a constant \(C_{\alpha,x}<\infty\), independent of the widths \(n_1,\dots,n_{L-1}\), such that for every \(\alpha\)-tuple \((i_1,\dots,i_\alpha)\) of pairwise distinct indices in \(\{1,\dots,n_L\}\),
\begin{equation}\label{eq:gaussian-hidden-layer-to-gp}
\left|
\E\Bigl[\Phi_\alpha\bigl(z^{(L;1)}_{i_1}(x),\dots,z^{(L;1)}_{i_\alpha}(x)\bigr)\Bigr]
-
\E\Bigl[\Phi_\alpha(G_1,\dots,G_\alpha)\Bigr]
\right|
\le
C_{\alpha,x}\sum_{k=1}^{L-1}\frac{1}{n_k},
\end{equation}
where \(G_1,\dots,G_\alpha\) are i.i.d. centered Gaussian random variables with variance \(\Kcal^{(L)}(x)\).
\end{lemma}

\begin{proof}
Since \(\sigma\in C_b^{3\cdot 2^{L-1}}(\R;\R)\), the function \(\Phi_\alpha\) is bounded and Lipschitz on \(\R^\alpha\).
Fix pairwise distinct indices \(i_1,\dots,i_\alpha\), and consider the random vector
\[
U_I(x):=\bigl(z^{(L;1)}_{i_1}(x),\dots,z^{(L;1)}_{i_\alpha}(x)\bigr).
\]
Because the rows of each Gaussian weight matrix are i.i.d., \(U_I(x)\) has the same law as the output at layer \(L\) of a fully Gaussian network of depth \(L\), evaluated at the single input \(x\), with output dimension \(\alpha\).
The corresponding Gaussian-process limit is the vector \((G_1,\dots,G_\alpha)\), whose coordinates are independent and distributed as \(\mathcal N(0,\Kcal^{(L)}(x))\).

Applying \cite[Theorem 4.2]{trevisan2023widedeepneuralnetworks} with \(p=1\), singleton input set \(X=\{x\}\), fixed output dimension \(\alpha\), and layer \(\ell=L\), we obtain
\[
\Wcal_1\bigl(U_I(x),(G_1,\dots,G_\alpha)\bigr)
\le
C_{\alpha,x}\sum_{k=1}^{L-1}\frac{1}{n_k}.
\]
Finally, by Kantorovich--Rubinstein duality,
\[
\left|\E[\Phi_\alpha(U_I(x))]-\E[\Phi_\alpha(G_1,\dots,G_\alpha)]\right|
\le
\|\Phi_\alpha\|_{\mathrm{Lip}}\,
\Wcal_1\bigl(U_I(x),(G_1,\dots,G_\alpha)\bigr),
\]
which gives \eqref{eq:gaussian-hidden-layer-to-gp}.
\end{proof}

\begin{proof}[Proof of \textbf{Theorem \ref{thm:W_2_bound}}]
From \textbf{Lemma \ref{lem:last_layer_replacement}}, it suffices to bound
\(\Wcal_2(\Zcal^{(L+1)}(x), z^{(L+1;L)}(x))\). From \textbf{Theorem \ref{th:entropic-bound}}, we have
\begin{equation}\label{eq:pre_w_2_bound}
    \Wcal_2(\Zcal^{(L+1)}(x), z^{(L+1;L)}(x))\le
    \frac{1}{\sqrt{\lambda(K)}}\,
    \Bigl(\E\|A-K\|_{HS}^2\Bigr)^{1/2}
    +
    \frac{2^{3/2}d^{1/4}}{\lambda(K)^{3/2}}\,
    \Bigl(\E\|A-K\|_{HS}^4\Bigr)^{1/2}.
\end{equation}
Here, 
\[
A = \frac{C_W}{n_{L}} \norm{\sigma(z^{L}(x))}_2^2 \text{I}_{n_{L+1}},
\]
which is the variance of \(z^{(L+1;L)}(x)\) conditioned on \(z^{(L)}(x)\), and \(\Kcal^{(L+1)(x)}\) is as in \textbf{Theorem \ref{thm:qualitative_thm}}.  

We have
\[
\begin{aligned}
\lefteqn{\E \left[ \norm{\frac{C_W}{n_{L}} \norm{\sigma(z^{L}(x))}_2^2 \text{I}_{n_{L+1}} - \Kcal^{(L+1)}(x) }_{HS}^p\right]} \qquad&\\
= & \E \left[ \norm{\frac{C_W}{n_{L}} \norm{\sigma(z^{L}(x))}_2^2 \text{I}_{n_{L+1}} - \E_{\Kcal^{(L)}(x)}[\sigma(\Zcal^{L}(x))^2] \text{I}_{n_{L+1}}}_{HS}^p\right] \\
=& \E \left[ \left( \sqrt{n_{L+1}} \left| \frac{C_W}{n_{L}} \norm{\sigma(z^{L}(x))}_2^2 - \E_{\Kcal^{(L)}(x)}[\sigma(\Zcal^{L}(x))^2] \right|\right)^p\right] \\
=& n_{L+1}^4 \E \left[ \left(  \left| \frac{C_W}{n_{L}} \norm{\sigma(z^{L}(x))}_2^2 - \E_{\Kcal^{(L)}(x)}[\sigma(\Zcal^{L}(x))^2] \right|\right)^p\right].
\end{aligned}
\]
For \(p\in\N\), and since \(\sigma\) is bounded, we use \textbf{Lemma \ref{lem:centered-average-moments}} to  get the upper bound
\begin{equation}\label{eq:HS_first bound}
    \E[\|A-K\|_{HS}^p]\lesssim\sum_{\alpha=1}^p
\left|
\frac{C_W^\alpha}{n_L^\alpha}
\sum_{i_1,\dots,i_\alpha=1}^{n_L}
\E\Bigl[\prod_{r=1}^\alpha \sigma( z^{(L)}_{i_k}(x) )^2\Bigr]
-
\E_{\Kcal^{(L)}(x)}[\sigma(\Zcal^{L}(x))^2]^\alpha
\right|.
\end{equation}
To conclude, we need to bound the right-hand side for \(p=2,4\). We write
\[
\begin{aligned}
    &\frac{C_W^\alpha}{n_L^\alpha} \sum_{i_1,\dots,i_\alpha=1}^{n_L} \E \left[ \prod_{k=1}^\alpha \sigma( z^{(L)}_{i_k}(x) )^2 \right]\\
    &= \sum_{j=0}^{L} \frac{C_W^\alpha}{n_L^\alpha} \sum_{i_1,\dots,i_\alpha=1}^{n_L} \left( \E \left[ \prod_{k=1}^\alpha \sigma( z^{(L; L-j)}_{i_k}(x) )^2 \right]- \E \left[ \prod_{k=1}^\alpha \sigma( z^{(L;L-j-1)}_{i_k}(x) )^2 \right] \right)\\
    &+\frac{C_W^\alpha}{n_L^\alpha} \sum_{i_1,\dots,i_\alpha=1}^{n_L} \E \left[ \prod_{k=1}^\alpha \sigma( z^{(L;1)}_{i_k}(x) )^2 \right]
\end{aligned}
\]
Since, \(\sigma\in C^{3\cdot 2^{L-1}}_b\), from \textbf{Theorem \ref{thm:DNN_Lindeberg_switching}}, and \textbf{Remark \ref{rem:Applicability}}, we get 
\begin{equation}\label{eq:first_corr_estimate}
    \left|\E \left[ \prod_{k=1}^\alpha \sigma( z^{(L; L-j)}_{i_k}(x) )^2 \right]- \E \left[ \prod_{k=1}^\alpha \sigma( z^{(L;L-j-1)}_{i_k}(x) )^2 \right]\right|\lesssim\frac{1}{\sqrt{n_{L-K-1}}}.
\end{equation}
Now, it is easy to see that  \textbf{Lemma \ref{lem:gaussian-hidden-layer-to-gp}}, implies that 
\begin{equation}\label{eq:second_corr_estimate}
    \left|\frac{C_W^\alpha}{n_L^\alpha} \sum_{i_1,\dots,i_\alpha=1}^{n_L} \E \left[ \prod_{k=1}^\alpha \sigma( z^{(L;1)}_{i_k}(x) )^2 \right]-\E_{\Kcal^{(L)}(x)}[\sigma(\Zcal^{L}(x))^2]^\alpha\right|\lesssim \sum_{k=1}^{L-1}\frac{1}{n_k}
\end{equation}
From \eqref{eq:first_corr_estimate} and \eqref{eq:second_corr_estimate}, we get the bound for \eqref{eq:HS_first bound}. This concludes the proof.
\end{proof}

The only thing left to prove is the following elementary lemma, which we used in the above proof 

\begin{lemma}\label{lem:centered-average-moments}
Let \(p\in\N\) be even, let \(n\in\N\), and let \(Y_1,\dots,Y_n\) be bounded real-valued random variables.
Set
\[
\overline Y_n:=\frac1n\sum_{i=1}^n Y_i,
\qquad
m:=\E[Y_1].
\]
Then there exists a constant \(C_p\), depending only on \(p\) and \(\sup_i\|Y_i\|_{L^\infty}\), such that
\[
\E\bigl[|\overline Y_n-m|^p\bigr]
\le
C_p
\sum_{\alpha=1}^p
\left|
\frac{1}{n^\alpha}
\sum_{i_1,\dots,i_\alpha=1}^{n}
\E\Bigl[\prod_{r=1}^\alpha Y_{i_r}\Bigr]
-
m^\alpha
\right|.
\]
\end{lemma}

\begin{proof}
Since \(p\) is even, \(|\overline Y_n-m|^p=(\overline Y_n-m)^p\).
By the binomial formula,
\[
(\overline Y_n-m)^p
=
\sum_{\alpha=0}^p \binom{p}{\alpha}(-m)^{p-\alpha}\,\overline Y_n^\alpha.
\]
Taking expectations gives
\[
\E[(\overline Y_n-m)^p]
=
\sum_{\alpha=0}^p \binom{p}{\alpha}(-m)^{p-\alpha}\,\E[\overline Y_n^\alpha].
\]
Since
\[
\E[\overline Y_n^\alpha]
=
\frac{1}{n^\alpha}
\sum_{i_1,\dots,i_\alpha=1}^n
\E\Bigl[\prod_{r=1}^\alpha Y_{i_r}\Bigr],
\]
and \(\E[\overline Y_n^0]=1=m^0\), we may rewrite
\[
\E[(\overline Y_n-m)^p]
=
\sum_{\alpha=0}^p \binom{p}{\alpha}(-m)^{p-\alpha}
\left(
\frac{1}{n^\alpha}
\sum_{i_1,\dots,i_\alpha=1}^n
\E\Bigl[\prod_{r=1}^\alpha Y_{i_r}\Bigr]
-
m^\alpha
\right),
\]
because
\[
\sum_{\alpha=0}^p \binom{p}{\alpha}(-m)^{p-\alpha}m^\alpha=(m-m)^p=0.
\]
Therefore
\[
\E[|\overline Y_n-m|^p]
\le
\sum_{\alpha=0}^p \binom{p}{\alpha}|m|^{p-\alpha}
\left|
\frac{1}{n^\alpha}
\sum_{i_1,\dots,i_\alpha=1}^n
\E\Bigl[\prod_{r=1}^\alpha Y_{i_r}\Bigr]
-
m^\alpha
\right|.
\]
The term \(\alpha=0\) vanishes identically, and since the \(Y_i\) are bounded, \(|m|\) is bounded by \(\sup_i\|Y_i\|_{L^\infty}\).
This yields the claim.
\end{proof}

\appendix
\section{Numerical Simulations}

The aim of this section is to illustrate numerically the convergence behaviour predicted by Theorem~\ref{thm:W_2_bound}. We consider fully connected networks with \(L=3\) hidden layers, scalar output, sigmoid activation, zero biases, and Laplace-distributed weights in the hidden layers, normalized to have variance \(C_W/n_{l-1}\). The first layer is Gaussian, in accordance with our assumptions.

For the convergence plots we take the common hidden width
\[
n_1=n_2=n_3=n \in \{16,32,64,128,256,512,1024,2048,4096,8192\}.
\]
For each value of \(n\), we generate independent samples of the network output and compare them with samples from the corresponding Gaussian limit. The Kolmogorov distance is estimated from the empirical distribution functions, while the \(W_1\) and \(W_2\) distances are computed from the empirical one-dimensional Wasserstein distances obtained by sorting the samples. The variance of the Gaussian limit is evaluated through the recursive covariance formula defining the infinite-width limit.

Figure~\ref{fig:convergence_distances} shows the empirical decay of the Kolmogorov, \(W_1\), and \(W_2\) distances as the width increases. Figure~\ref{fig:distribution_overlay} compares the output distribution at a fixed large width with its Gaussian approximation. The numerical results are consistent with the quantitative convergence predicted by our theory.

To reduce Monte Carlo noise, the number of realizations was increased with the width, and each experiment was repeated several times.

More precisely, for \(n_{\min}=16\), we used \(M(n)=800\,(n/n_{\min})\) realizations and \(R=6\) independent repetitions.

\begin{figure}[htpb!]
    \centering
    \includegraphics[width=0.98\linewidth]{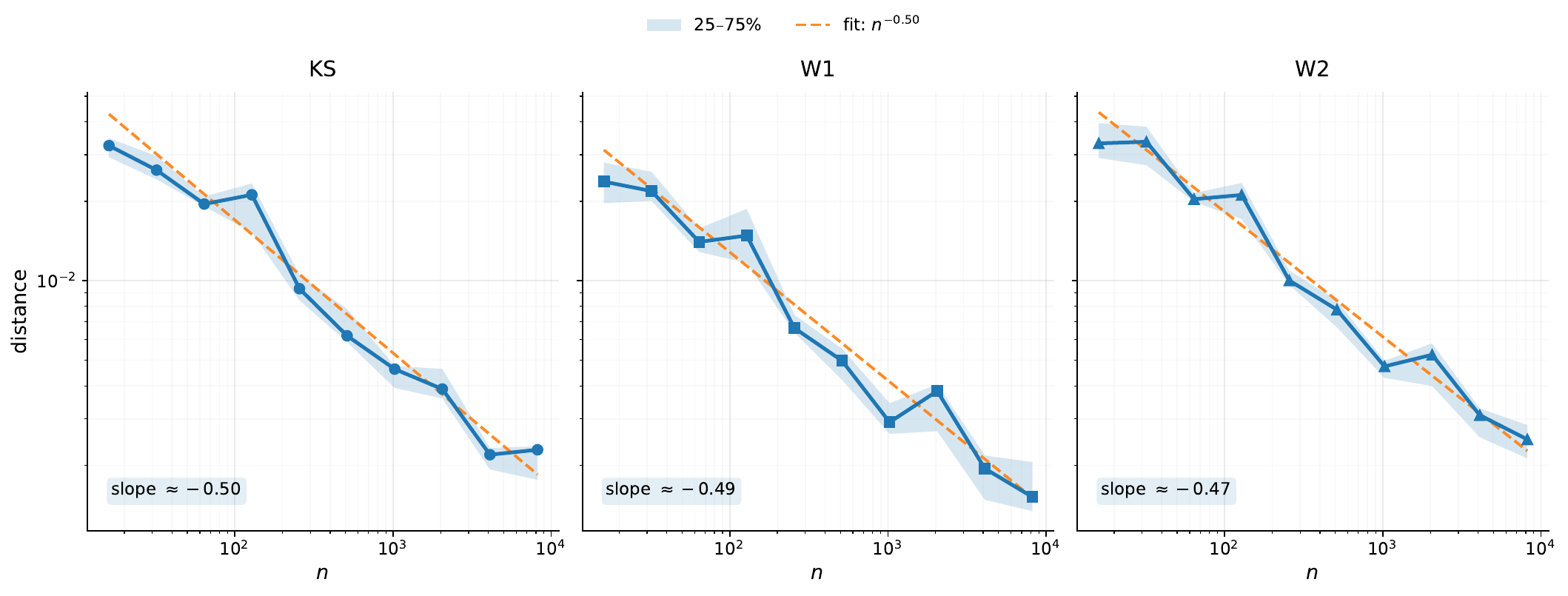}
    \caption{Empirical decay of the Kolmogorov, \(W_1\), and \(W_2\) distances as the hidden width increases.}
    \label{fig:convergence_distances}
\end{figure}

\begin{figure}[htpb!]
    \centering
    \includegraphics[width=0.65\linewidth]{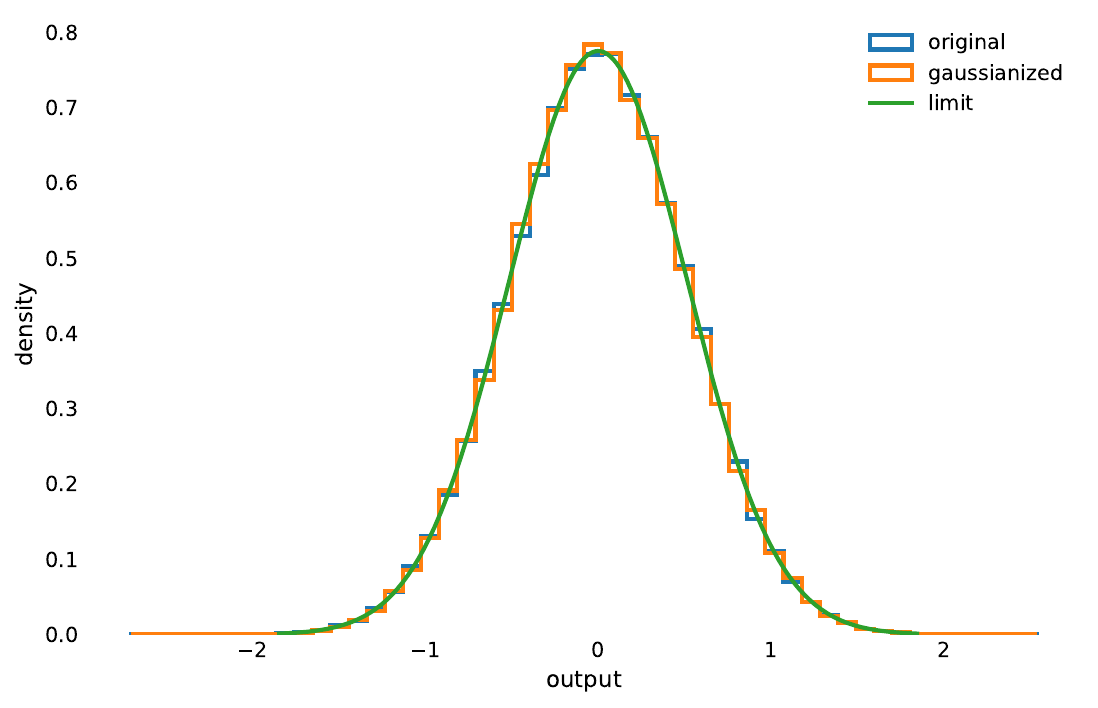}
    \caption{Comparison of the output distribution of the finite-width network with its Gaussian approximation at a fixed large width.}
    \label{fig:distribution_overlay}
\end{figure}

\section{Proof of \textbf{Theorem \ref{th:entropic-bound}}}\label{app:Cond_Gauss_proof}
Here we prove \textbf{Theorem \ref{th:entropic-bound}}. The proof  is inspired by \cite{trevisan2023widedeepneuralnetworks}, and a version of it has been proven in \cite{entr_bounds_cond_gaussian}.

\begin{remark}\label{rem:condi-gauss-explicit}
Let \(X\) be a centered random vector in \(\R^d\), and let \(\Fcal\) be a \(\sigma\)-field such that, conditionally on \(\Fcal\), the law of \(X\) is Gaussian with covariance matrix \(M\).
Then \(X\) has the same law as \(\sqrt M\,N\), where \(N\sim \mathcal N_d(0,I_d)\) is independent of \(M\).
\end{remark}

\begin{proof}[Proof of \Cref{th:entropic-bound}]
Since the value of \(W_2\) is bounded above by the cost of any coupling, we may realize \(F\) and \(G\) on the same probability space using the same Gaussian vector \(N\).
Then
\[
W_2(F,G)^2
\le
\E\|\sqrt A\,N-\sqrt K\,N\|^2
=
\E\|\sqrt A-\sqrt K\|_{HS}^2.
\]

Set
\[
\lambda:=\lambda(K)>0,
\qquad
E:=\{\|A-K\|_{op}\le \lambda/2\}.
\]
We split
\[
\E\|\sqrt A-\sqrt K\|_{HS}^2 = T_1+T_2,
\]
where
\[
T_1:=\E\!\left[\|\sqrt A-\sqrt K\|_{HS}^2\,\one_E\right],
\qquad
T_2:=\E\!\left[\|\sqrt A-\sqrt K\|_{HS}^2\,\one_{E^c}\right].
\]

On the event \(E\), a standard perturbation bound for the matrix square root gives
\[
\|\sqrt A-\sqrt K\|_{HS}
\le
\frac{1}{\sqrt\lambda}\|A-K\|_{HS}.
\]
Therefore
\[
T_1\le \frac{1}{\lambda}\E\|A-K\|_{HS}^2.
\]

On \(E^c\), the Powers--St\o rmer inequality implies
\[
\|\sqrt A-\sqrt K\|_{HS}^2\le \|A-K\|_{S_1}\le \sqrt d\,\|A-K\|_{HS}.
\]
Since on \(E^c\),
\[
\|A-K\|_{HS}\ge \|A-K\|_{op}>\frac{\lambda}{2},
\]
it follows that
\[
\one_{E^c}
\le
\left(\frac{2\|A-K\|_{HS}}{\lambda}\right)^3.
\]
Hence
\[
T_2
\le
\frac{8\sqrt d}{\lambda^3}\E\|A-K\|_{HS}^4.
\]
Combining the bounds for \(T_1\) and \(T_2\), we obtain
\[
W_2(F,G)^2
\le
\frac{1}{\lambda}\E\|A-K\|_{HS}^2
+
\frac{8\sqrt d}{\lambda^3}\E\|A-K\|_{HS}^4.
\]
Taking square roots and using \(\sqrt{x+y}\le \sqrt x+\sqrt y\) concludes the proof.
\end{proof}

\section*{Acknowledgements}
FG acknowledges support from Imperial College London through the Roth PhD scholarship, and his PhD supervisor Dan Crisan for his continuous guidance. SK and MR acknowledge the support of the project \emph{Noise in fluid dynamics and related models} funded by the MUR Progetti di Ricerca di Rilevante Interesse Nazionale (PRIN) Bando 2022 - grant 20222YRYSP. MR acknowledges the partial support of the project PNRR - M4C2 - Investimento 1.3, Partenariato Esteso PE00000013 - \emph{FAIR - Future Artificial Intelligence Research} - Spoke 1 \emph{Human-centered AI}, funded by the European Commission under the NextGeneration EU programme, and of the MUR Excellence Department Project awarded to the Department of Mathematics, University of Pisa, CUP I57G22000700001.

\bibliographystyle{amsalpha}
\bibliography{mybibliography}
\end{document}